\documentclass[12pt]{amsart}
\usepackage{amsmath, amssymb, amsthm, latexsym}
\usepackage{tikz-cd}
\usepackage{amssymb,amscd,amsmath}
\usepackage{latexsym}
\usepackage[center]{caption}
\usepackage{tikz}
\usepackage{tikz}
\usepackage{mathrsfs}
\newcounter{braid}
\newcounter{strands}

\DeclareMathAlphabet{\bsf}{OT1}{cmss}{bx}{n}

\pgfkeyssetvalue{/tikz/braid height}{1cm}
\pgfkeyssetvalue{/tikz/braid width}{1cm}
\pgfkeyssetvalue{/tikz/braid start}{(0,0)}
\pgfkeyssetvalue{/tikz/braid colour}{black}
\pgfkeys{/tikz/strands/.code={\setcounter{strands}{#1}}}

\makeatletter
\def\cross{%
  \@ifnextchar^{\message{Got sup}\cross@sup}{\cross@sub}}

\def\cross@sup^#1_#2{\render@cross{#2}{#1}}

\def\cross@sub_#1{\@ifnextchar^{\cross@@sub{#1}}{\render@cross{#1}{1}}}

\def\cross@@sub#1^#2{\render@cross{#1}{#2}}

\def\render@cross#1#2{
  \def\strand{#1}
  \def\crossing{#2}
  \pgfmathsetmacro{\cross@y}{-\value{braid}*\braid@h}
  \pgfmathtruncatemacro{\nextstrand}{#1+1}
  \foreach \thread in {1,...,\value{strands}}
  {
    \pgfmathsetmacro{\strand@x}{\thread * \braid@w}
    \ifnum\thread=\strand
    \pgfmathsetmacro{\over@x}{\strand * \braid@w + .5*(1 - \crossing) * \braid@w}
    \pgfmathsetmacro{\under@x}{\strand * \braid@w + .5*(1 + \crossing) * \braid@w}
    \draw[braid] \pgfkeysvalueof{/tikz/braid start} +(\under@x pt,\cross@y pt) to[out=-90,in=90] +(\over@x pt,\cross@y pt -\braid@h);
    \draw[braid] \pgfkeysvalueof{/tikz/braid start} +(\over@x pt,\cross@y pt) to[out=-90,in=90] +(\under@x pt,\cross@y pt -\braid@h);
    \else
    \ifnum\thread=\nextstrand
    \else
     \draw[braid] \pgfkeysvalueof{/tikz/braid start} ++(\strand@x pt,\cross@y pt) -- ++(0,-\braid@h);
    \fi
   \fi
  }
  \stepcounter{braid}
}

\tikzset{braid/.style={double=\pgfkeysvalueof{/tikz/braid colour},double distance=1pt,line width=2pt,white}}

\newcommand{\braid}[2][]{%
  \begingroup
  \pgfkeys{/tikz/strands=2}
  \tikzset{#1}
  \pgfkeysgetvalue{/tikz/braid width}{\braid@w}
  \pgfkeysgetvalue{/tikz/braid height}{\braid@h}
  \setcounter{braid}{0}
  \let\sigma=\cross
  #2
  \endgroup
}
\makeatother

\input xypic
\newtheorem{theorem}{Theorem}
\newtheorem{proposition}[theorem]{Proposition}

\newtheorem{lemma}[theorem]{Lemma}

\newtheorem{corollary}[theorem]{Corollary}


\def\Z{\mathbb{Z}}

\def\md{\mathcal{D}}

\def\qed{\hfill$\square$\medskip}

\def\Zpk{\mathbb{Z}/p^{k}}
\def\Zpk1{\mathbb{Z}/p^{k-1}}

\newcommand{\rref}[1]{(\ref{#1})}

\newcommand{\beg}[2]{\begin{equation}\label{#1}#2\end{equation}}

\def\sl2{\widetilde{SL_{2}(\Z)}}

\def\md
\def\rank{\operatorname{rank}}

\title[]{Notes on equivariant homology with constant coefficients}
\author{Sophie Kriz}

\begin{document}
\maketitle

\vspace{5mm}

\begin{abstract}
In this paper, for a finite group, we discuss a method for calculating equivariant homology
with constant coefficients. We apply it to completely calculate the geometric fixed points of the equivariant spectrum
representing equivariant (co)homology with constant coefficients. We also treat a more complicated example of inverting the standard
representation in the equivariant homology of split extraspecial groups at the prime 2.
\end{abstract}

\vspace{5mm}

\section{Introduction}

\vspace{5mm}

Equivariant spectra are the foundation of equivariant generalized homology and cohomology theory of $G$-spaces (and ultimately, $G$-spectra) for a finite (or more generally compact Lie) group $G$,
which has all the formal properties of generalized non-equivariant homology and cohomology theory, including duality, along with stability under suspensions by
finite-dimensional real representations. They were introduced and developed in \cite{LewisMaySteinbergerEquivariant} (see also \cite{AdamsPrerequisites, GreenleesAdamsSpectralSequence}).
Equivariant homology and cohomology $H\underline{A}_G$ with constant coefficients in an abelian group $A$, on the other
hand, can be defined on the chain level, as a part of the theory of Bredon \cite{Bredon}.
Both contexts are reconciled in \cite{LewisMayMcClureOrdinary}, where, more generally, equivariant Eilenberg-MacLane spectra of Mackey functors are defined.
In this paper, we discuss a spectral sequence (Proposition \ref{spectralsequence}, \ref{spectralsequencebased}) which can be used to 
compute generalized equivariant homology of a $G$-CW-complex from its subquotients of constant isotropy.
This spectral sequence is especially efficient in the case of $H\underline{A}_G$.
For example, we shall prove that for a (finite) $p$-group $G$, the spectral sequence computing $H\underline{\Z/p}_*^G (X)$
always collapses to $E^1$ (see Theorem \ref{spectralsequencecollapses} below).

Note that this is false in cohomology.
By \cite{Bredon}, for a $G$-CW-complex $X$ and an abelian group $A$,
\beg{bredonformula}{H^*_G(X; \underline{A}) = H^* (X/G; A).}
Let $G= \Z/2$ and $X= S^{\alpha}$ where $\alpha$ denotes the sign representation of $\Z /2$.
Then $X / (\Z/2) \simeq *$.
Thus, by \rref{bredonformula}, $H_G^n(X; \underline{A})$ is only non-trivial for $n=0$, while $X$ has a $1$-cell of isotropy $\{ e\}$.

One may think, therefore, that equivariant homology with constant coefficients carries less information than cohomology.
It turns out, however, that equivariant $E$-homologies of certain spaces give important information about a spectrum $E$.
For example, the {\em geometric fixed point spectra}
(\cite{GreenleesAdamsSpectralSequence} and \cite{LewisMaySteinbergerEquivariant}, II \textsection 8, 9)
$\Phi^HE$, where $H$ runs through subgroups of $G$, completely characterize the spectrum $E$.
The coefficients $\Phi^G E_*$, for a finite group $G$, are the reduced $E$-homology of the smash product
$S^{\infty V}$ of infinitely many copies of the one-point compactification $S^V$ of the reduced regular representation $V$.
Geometric fixed point spectra proved very useful in applications, for example, in \cite{DieckBordism, HesselholtMadsenKTheory, HillHopkinsRavenelKervaire}.

We will see that our method allows a complete computation of the coefficients of the geometric fixed point spectrum of homology with constant coefficients $\Phi^G (H\underline{A})_*$, which we denote by $\Phi^G(\underline{A})_*$,
by reducing it to the case where $G$ is an elementary abelian group.
We show (Proposition \ref{nonabelianpgp}) that $\Phi^G H\underline{A}_G=0$ if $G$ is not a $p$-group, and that
$$\Phi^G H\underline{A}_G = \Phi^{G/G_p'} H\underline{A}_{G/G_p'}$$
where $G$ is a $p$-group and $G_p'$ is its Frattini subgroup (Proposition \ref{abelianpgp}).
In the elementary abelian case, the computation was carried out for $A=\Z/p$ in my previous paper \cite{SophieKrizEquivariant} (see also \cite{HollerIgorKrizCoefficients, HollerIgorKrizOrdinary}).

Let first $p=2$.
We have
$$H^*((\Z/2)^n ;\Z/2) = \Z/2[x_1,\dots, x_n]$$
where $x_i$ have cohomological dimension $1$.
Let, for $\alpha = (\alpha_1, \dots, \alpha_n) \in (\Z/2)^n \smallsetminus \{ 0\}$,
$$x_{\alpha} = \alpha_1 x_1 + \dots \alpha_n x_n.$$

For $p>2$, following the notation of \cite{SophieKrizEquivariant}, we have
$$ H^* ((\Z/p)^n ; \Z/p) = \Z/p [z_i] \otimes \Lambda_{\Z/p} [dz_i]$$
where $z_i$ have cohomological dimension $2$ and $dz_i$ have cohomological dimension $1$.
Let, for $\alpha = (\alpha_1,\dots, \alpha_n)\in (\Z/p)^n \smallsetminus \{0\}$,
$$z_\alpha = \alpha_1 z_1 + \dots + \alpha_n z_n,$$
$$dz_\alpha = \alpha_1 dz_1 + \dots + \alpha_n d z_n.$$

\begin{theorem}\label{Prepretheorem} (\cite{ HollerIgorKrizCoefficients, HollerIgorKrizOrdinary, SophieKrizEquivariant})
(a) $\Phi^{(\Z/2)^n} (\underline{\Z/2})$ is the subring of 
$$H^* ((\Z/2)^n;\Z/2) [x_{\alpha}^{-1}\mid \alpha \neq 0]$$
generated by $y_{\alpha} = x_{\alpha}^{-1}$.

(b) For $p>2$, $\Phi^{(\Z / p)^n} (\underline{\Z/p})$ is the subring of
$$H^* ((\Z /p)^n; \Z / p) [z^{-1}_\alpha | \alpha \in (\Z / p )^n \smallsetminus \{0\}]$$
generated by $t_\alpha = z_\alpha^{-1}$ and $u_\alpha = t_\alpha dz_\alpha$.

\end{theorem}

One can also explicitly describe these rings in generators and defining relations:
\begin{theorem} \label{Pretheorem} (\cite{ HollerIgorKrizCoefficients, HollerIgorKrizOrdinary, SophieKrizEquivariant})
(a) For $p=2$, we have
$$\Phi^{(\Z/2)^n}_* (\underline{\Z/2}) \cong \Z/2 [ y_\alpha | \alpha \in (\Z/2)^n \smallsetminus \{ 0 \}] / \sim $$
where $\sim$ denotes the relations
$$y_\alpha y_\beta + y_\alpha y_\gamma +
y_\beta y_\gamma \sim 0 $$
for  $\alpha + \beta + \gamma = 0$.
The elements $y_\alpha$ are in degree $1$. 

(b) For $p>2$,
\beg{oldresult}{\Phi_*^{(\Z/ p )^n} (\underline{\Z/p}) = \Z/p [t_{\alpha}] \otimes \Lambda_{\Z/p} [u_\alpha]/\sim}
where $\sim$ denotes the relations
$$t_{i\alpha} \sim i^{-1} t_{\alpha}, \; u_{i \alpha} \sim u_{\alpha}$$
$$t_\beta t_{\alpha + \beta} + t_\alpha t_{\alpha + \beta}\sim t_\alpha t_\beta$$
$$t_\beta u_{\alpha + \beta } - t_{\alpha + \beta} u_\beta + t_{\alpha + \beta} u_{\alpha} \sim u_\alpha t_{\beta}$$
$$-u_{\beta} u_{\alpha + \beta} + u_{\alpha} u_{\alpha + \beta} \sim u_\alpha u_\beta,$$
where $i \in \Z/p \smallsetminus \{0\}$, for $\alpha, \beta, \alpha+ \beta \in (\Z/p)^n \smallsetminus \{ 0 \}$.
The elements $u_\alpha$ are in degree $1$ and the elements $t_\alpha$ are of degree $2$. 
\end{theorem}

This allows a complete answer for $\underline{\Z}$, the universal constant coefficients.
Fix an element $\alpha_0 \in (\Z/p)^n \smallsetminus \{ 0\}$
and put $\widetilde{u}_\alpha = u_\alpha - u_{\alpha_0}$, ($\widetilde{y}_\alpha = y_\alpha - y_{\alpha_0}$ for $p=2$). For
$p=2$, we also set $t_{\alpha_0} = y^2_{\alpha_0}$.

\begin{theorem}\label{IntegralCase}
(a) For every prime $p$, $\Phi_*^{(\Z / p)^n}(\underline{\Z})$ is the subring of the ring $\Phi_*^{(\Z/p)^n}(\underline{\Z/p})$ on which
the Bockstein
$$\beta: \Phi_*^{(\Z / p)^n} (\underline{\Z / p})\rightarrow \Phi_{*-1}^{(\Z/p)^n} (\underline{\Z/p})$$ 
vanishes.

(b) Explicitly, for $p=2$, we have
$$\Phi^{(\Z/2)^n}_* (\underline{\Z}) \cong \Z / 2 [ \widetilde{y}_\alpha , t_{\alpha_0} | \alpha \in (\Z/2)^n \smallsetminus \{0 \}] /\sim$$
where $\sim$ denotes the relations
$$\widetilde{y}_{\alpha_0} \sim 0$$
$$\widetilde{y}_\alpha \widetilde{y}_\beta + \widetilde{y}_\alpha \widetilde{y}_\gamma + \widetilde{y}_\beta \widetilde{y}_\gamma + t_{\alpha_0}
\sim 0$$
where $\alpha +\beta + \gamma =0$.

(c) For $p>2$
$$\begin{array}{c}
\Phi^{(\Z /p)^n}_* (\underline{\Z}) \cong (\Z /p [t_\alpha |\alpha \in (\Z/p)^n \smallsetminus \{ 0\}] \\
\otimes \Lambda_{\Z/p} [\widetilde{u}_\alpha | \alpha \in (\Z/p)^n \smallsetminus \{ 0\}] )/\sim
\end{array}
$$
where $\sim$ denotes the relations
$$t_{i\alpha} \sim i^{-1} t_{\alpha}, \; \widetilde{u}_{i\alpha} \sim \widetilde{u}_\alpha$$
$$t_\beta t_{\alpha + \beta} + t_\alpha t_{\alpha + \beta} \sim t_\alpha t_\beta$$
\beg{relationspg2}{\widetilde{u}_{\alpha_0} \sim 0}
$$t_\beta \widetilde{u}_{\alpha + \beta} - t_{\alpha + \beta} \widetilde{u}_\beta + t_{\alpha + \beta} \widetilde{u}_\alpha \sim t_\beta \widetilde{u}_\alpha$$
$$ - \widetilde{u}_\beta \widetilde{u}_{\alpha+\beta} + \widetilde{u}_\alpha \widetilde{u}_{\alpha + \beta} \sim \widetilde{u}_\alpha \widetilde{u}_\beta,$$
for $i\in \Z/ p \smallsetminus \{ 0\}$ and $\alpha, \beta, \alpha + \beta \in (\Z /p )^n \smallsetminus \{ 0\}$.
\end{theorem}

The reductions contained in Proposition \ref{abelianpgp} and Proposition \ref{nonabelianpgp} are quite easy.
However, if one considers the more general problem of calculating $\widetilde{H\underline{A}}_*^G (S^{\infty \gamma})$
for a general finite dimensional representation $\gamma$ of $G$, one gets non-trivial examples.
One such example is treated in Section \ref{ExtraSpecialExample}, where $G$ is a split extraspecial $2$-group and
$\gamma$ is the irreducible representation non-trivial on the center.

The present paper is organized as follows:
In Section \ref{MainMethod}, we discuss our spectral sequences.
In Section \ref{ProofsMainResult}, we discuss the application to geometric fixed points. Section \ref{ExtraSpecialExample} contains the extraspecial group example.

\vspace{5mm}

\section{The Spectral Sequences}\label{MainMethod}

\vspace{5mm}

For a $G$-equivariant spectrum $E$, we will need to use the homotopy co-fixed point (Borel homology) spectrum
\beg{BorelSpectrumDefn}{E_{hG} = (E\wedge EG_+)^G}
where $EG$ is a non-equivariantly contractible free $G$-CW complex and for a space $X$, we write
$X_+ = X \coprod \{*\}$.
The formula \rref{BorelSpectrumDefn} includes a key fact called the Adams isomorphism (\cite{LewisMaySteinbergerEquivariant} II \textsection 7):
$G$-equivariant cell spectra with $G$-free cells can be identified with naive (i.e. non-equivariant) cell spectra with a free cellular $G$-action.
The Adams isomorphism says that for any cell $G$-spectrum $E$,
$E_{hG}$ is equivalent to $(E\wedge EG_+)/G$ where $E\wedge EG_+$ is considered
as a naive $G$-spectrum.

Recall that a {\em family} $\mathscr{F}$ is defined as a set of subgroups of $G$ that is closed under sub-conjugacies.
For a family $\mathscr{F}$, we have a $G$-CW complex $E\mathscr{F}$ such that
$$\begin{array}{c}
E\mathscr{F}^H \simeq * \text{, for } H\in \mathscr{F}\\
E\mathscr{F}^G \simeq \emptyset \text{, for } H\notin \mathscr{F}.\\
\end{array}$$
If we denote by $\widetilde{X}$ the unreduced suspension of a $G$-space $X$, we have
$$\begin{array}{c}
\widetilde{E\mathscr{F}}^H \simeq * \text{, for } H\in \mathscr{F}\\
\widetilde{E\mathscr{F}}^G \simeq S^0\text{, for } H\notin \mathscr{F}.\\
\end{array}$$

If $V$ is a real $G$-representation, denote by $S(V)$ the unit sphere of $V$
and by $S^V$ the union of the $1$-point compactifications of $S^W$ for finite dimensional subrepresentations $W$ of $V$.
If we set $\infty V = \bigoplus_{\infty} V$, then $S(\infty V)$ is a model for $E\mathscr{F}_V$
where $\mathscr{F}_V = \{ H\subseteq G | V^H \neq 0\}$.
Thus $S^{\infty V}$ is a model for $\widetilde{E\mathscr{F}}_V$.
Since $S^{\infty V} \wedge S^{\infty V} = S^{\infty V}$, for a commutative ring spectrum $E$,
$$\widetilde{E}_* \widetilde{E\mathscr{F}}_V = \widetilde{E}_* S^{\infty V}$$
is a commutative ring.
Here $\widetilde{E}_* X$, for a $G$-spectrum $E$ and a based $G$-CW complex $X$ (recall that the base point is
required to be $G$-fixed), is the equivariant reduced homology of $X$, i.e. $\pi_* (E\wedge \Sigma^{\infty} X)$
(without adding a disjoint base point).
Note that this is also the $\Z$-graded part of the $RO(G)$-graded coefficient ring of $\alpha_V^{-1} E$ where
$\alpha_V \in \pi_{-V} E$ is the class obtained from the inclusion $S^0 \rightarrow S^V$.

For a finite group $G$, a family $\mathscr{F}$, and an $H\in \mathscr{F}$, define the height of $H$ inductively by
$$h_{\mathscr{F}}(H) = \text{max} \{0, h_{\mathscr{F}} (K) | K\in \mathscr{F}, H\subsetneq K\}+1.$$

Now let $X$ be a $G$-CW-complex. Consider the family
$$\mathscr{F}=\mathscr{F}_X = \{ H\subseteq G | X^H \neq \emptyset \}.$$
Let $E$ be a $G$-spectrum.
Then we have a spectral sequence converging to the $E$-homology of $X$, using Borel homology of parts of $X$ of the same isotropy.
There are two versions of the spectral sequence, one for the unreduced homology of $X$, the other for the reduced homology
of its unreduced suspension. Keeping track of terms can be delicate, so we list both versions:

\begin{proposition}\label{spectralsequence}
We have a spectral sequence
$$E^1_{p,q} \Rightarrow E_{p+q} X$$
where
\beg{spectralsequenceformula}{E^1_{p,q} = \bigoplus_{(H), H\in \mathscr{F}_X, h_{\mathscr{F}_X}(H)=p} ((E^H \wedge (X^H / \bigcup_{H\subsetneq K} X^K))_{hW(H)})_{p+q}.}
(Here $W(H)=N(H)/H$ where $N(H)$ is the normalizer of $H$ in $G$, and $(H)$ runs through the conjugacy classes of $H\in \mathscr{F}_X$.)
\end{proposition}

\begin{proposition}\label{spectralsequencebased}
We have a spectral sequence
$$E^1_{p,q} \Rightarrow \widetilde{E}_{p+q} \widetilde{X}$$
where
$$E_{0,q}^1 =E_q(*)$$
\beg{spectralsequencebasedformula}{E_{p,q}^1 =\bigoplus_{(H), H\in \mathscr{F}_X, h_{\mathscr{F}_X}(H)=p} ((E^H \wedge (X^H/ \bigcup_{H\subsetneq K} X^K))_{h W(H)})_{p+q-1}.}
\end{proposition}

\vspace{5mm}

The proofs of these statements are sufficiently similar to only give one of them.
We prove Proposition \ref{spectralsequencebased} which is more closely related to our applications.

\vspace{5mm}

\noindent {\em Proof of Proposition \ref{spectralsequencebased}:}
Define an increasing $G$-equivariant filtration of $X$
by
$$F'_p X = \bigcup_{h_{\mathscr{F}}(H)\leq p }X^H = \bigcup_{h_{\mathscr{F}}(H)= p } X^H.$$
Then define an increasing $G$-equivariant filtration of $\widetilde{X}$ by
$$F_0 \widetilde{X} = S^0$$
$$F_p \widetilde{X} = \bigcup_{h_{\mathscr{F}}(H)\leq p} (\widetilde{X})^H = \bigcup_{h_{\mathscr{F}}(H) = p}  (\widetilde{X})^H .$$

We have a spectral sequence
$$E^1_{p, q} = \widetilde{E}_{p+q} ( F_p \widetilde{X} / F_{p-1} \widetilde{X}) \Rightarrow \widetilde{E}_{p+q } (\widetilde{X}) = (E\wedge\widetilde{X})_{p+q}.$$
By definition, for $p\geq 1$, 
$$F_p \widetilde{X} / F_{p-1} \widetilde{X} = \bigcup_{h_{\mathscr{F}}(H) = p } (\widetilde{X})^H/\bigcup_{h_{\mathscr{F}}(H) \leq p-1} (\widetilde{X})^H.$$
On the other hand,
$$F_p \widetilde{X} / F_{p-1} \widetilde{X}= (F_p \widetilde{X} / F_0 \widetilde{X}) )/  (F_{p-1} \widetilde{X} / F_0 \widetilde{X})=$$
$$= (\Sigma F'_p (X)_+ )/ (\Sigma F'_{p-1} (X)_+) = $$
$$=\Sigma  (F'_p (X)/ F'_{p-1} (X)).$$
Thus,
$$\widetilde{E}_{p+q}(F_p \widetilde{X} / F_{p-1} \widetilde{X})= \widetilde{E}_{p+q-1} (F'_p (X)/ F'_{p-1} (X)).$$

Now, we have
$$F'_p (X)/ F'_{p-1} (X) = \bigcup_{h_{\mathscr{F}}(H)=p} X^H / \bigcup_{h_{\mathscr{F}}(H)<p} X^H=$$
$$= \bigvee_{h_{\mathscr{F}}(H) = p} \left( X^H / \bigcup_{H\subsetneq K\in\mathscr{F}} X^K \right).$$
Note that 
$$X^H / \bigcup_{H\subsetneq K\in \mathscr{F}} X^K$$ 
is a free based $W(H)$-CW complex.

On the other hand,
$$\widetilde{E}_{p+q -1}^G (\bigvee_{h_{\mathscr{F}}(H) = p} (X^H / \bigcup_{H\subsetneq K\in \mathscr{F}} X^K))=$$
$$= \bigoplus_{(H), \; h_{\mathscr{F}}(H)= p} \widetilde{E}_{p+q-1}^G (\bigvee_{H' = gHg^{-1}} (X^H / \bigcup_{H'\subsetneq K\in \mathscr{F}} X)^K) )=$$
$$ = \bigoplus_{(H), \; h_{\mathscr{F}}(H) = p} \widetilde{E}^{N(H)}_{p+q-1} (X^H / \bigcup_{H\subsetneq K \in \mathscr{F}}X^K).$$

To justify the third isomorphism above, note that $G$ acts transitively on the conjugacy classes of $H$
and thus
$$\bigvee_{H' = gHg^{-1}} (X^H / \bigcup_{H'\subsetneq K\in \mathscr{F}} X)^K) $$
is the pushforward from $N(H)$ to $G$ of
$$X^H / \bigcup_{H\subsetneq K \in \mathscr{F}}X^K.$$

Therefore,
$$E^1_{p, q} = \bigoplus_{(H), \; h_{\mathscr{F}}(H) = p} \widetilde{E}^{N(H)}_{p+q-1} ( (X^H/\bigcup_{H\subsetneq K} X^K) \wedge EW(H)_+)=$$
$$=\bigoplus_{(H), \; h_{\mathscr{F}}(H) = p} ((E^H \wedge (X^H/\bigcup_{H\subsetneq K} X^K ))_{hW(H)})_{p+q-1},$$
since $E^G = (E^K )^{G/K}$.

\qed

\vspace{5mm}

We shall be especially interested in the case of classifying spaces of families.
Consider, for a family $\mathscr{F}$ and a group $H\in \mathscr{F}$, the poset
\beg{POSETPGHDefn}{P^{\mathscr{F}}_H=\{ K \in \mathscr{F} | K\supsetneq H\}}
with respect to inclusion. Note that this poset has a $N(H)$-action by conjugation.

For a poset $P$, denote by $|P|$ the nerve (also called the classifying space or bar construction) of $P$, which one
defines as the geometric realization of the simplicial set whose $n$-simplicies are chains
$$x_0 \leq \dots \leq x_n$$
where faces are given by deletions and degeneracies by repetitions.
This is a special case of the nerve of a category where $n$-simplicies are composable $n$-tuples of morphisms.

\vspace{5mm}

\begin{corollary}
We have a spectral sequence
$$E^1_{p,q} \Rightarrow (E\wedge \widetilde{E\mathscr{F}})_{p+q} = \widetilde{E}_{p+q} \widetilde{E\mathscr{F}}$$
where
$$E^1_{0,q}=E_q(*)$$
and for $p>0$,
\beg{Specialspectralsequencebased}{E^1_{p,q} = \bigoplus_{(H), \; H\in \mathscr{F}, \\ h_{\mathscr{F}}(H)= p} ((E^H \wedge \widetilde{|P^{\mathscr{F}}_H|})_{h W(H)})_{p+q-1}}
where $(H)$ runs through the conjugacy classes of groups $H\in \mathscr{F}$.
\end{corollary}

\begin{proof}
Apply Proposition \ref{spectralsequencebased} to $X=E\mathscr{F}$. Note that $\mathscr{F}_{E\mathscr{F}} = \mathscr{F}$.
Let $H\in \mathscr{F} = \mathscr{F}_X$. Then 
$$E\mathscr{F}^H \simeq *.$$
We may realize
the system of spaces $(E\mathscr{F}^K)_{H\subsetneq K \in \mathscr{F}}$ as a $N(H)$-equivariant functor
$$F:P_H^{\mathscr{F}} \rightarrow \Delta^{Op} \text{-}Set$$
where the right hand side denotes the category of simplicial sets, such that the canonical maps
$$\operatornamewithlimits{colim}_{y<x} F(y) \rightarrow F(x)$$
are injective for all $x\in P_H^{\mathscr{F}}$.
Now if we denote by $\pi: P_H^{\mathscr{F}} \rightarrow *$ the terminal map, $\bigcup_{H\subsetneq K \in \mathscr{F}} E\mathscr{F}^K$
is the left Kan extension $\pi_{\#} F$.
However, by our injectivity assumption, the canonical $N(H)$-equivariant map
\beg{LeftKanExtpi}{L\pi_{\#} F \rightarrow \pi_{\#} F}
(where $L$ denotes the left derived functor)
is a non-equivariant equivalence.
Additionally, the left hand side can be expressed as the $2$-sided bar construction
$B(*, P_H^{\mathscr{F}}, F)$.
Since the values of $F$ on objects are contractible, we have an $N(H)$-equivariant map
\beg{BarConstructionMap}{B(*, P_H^{\mathscr{F}}, F) \rightarrow B(*, P_H^{\mathscr{F}}, *) = |P_H^{\mathscr{F}}|}
which is a non-equivariant equivalence.
The equivariant maps \rref{LeftKanExtpi}, \rref{BarConstructionMap} induce equivalences on homotopy fixed points, since
they are non-equivariant equivalences.
Thus, we have an equivalence
$$(E^H \wedge (E\mathscr{F}^H / \bigcup_{H\subsetneq K} E\mathscr{F}^K))_{hW(H)} \sim (E^H \wedge \widetilde{|P^{\mathscr{F}}_H|})_{h W(H)}$$
as required.


\end{proof}

Again, there is also an unbased version for $E\mathscr{F}$ instead of $\widetilde{E\mathscr{F}}$.

\vspace{5mm}

\begin{theorem}\label{spectralsequencecollapses}
If $G$ is a $p$-group and $E=H\underline{\Z/p}$, then the spectral sequence \rref{spectralsequenceformula} collapses to $E^1$.
Additionally, the spectral sequence \rref{spectralsequencebasedformula} collapses to $E^1$ when $X^G = \emptyset$.
\end{theorem}

\begin{proof}
Again, the proofs of the reduced and unreduced cases are similar. We treat the unreduced case this time.

Suppose $G$ is a $p$-group and $X$ is a $G$-CW-complex.
Then $H\underline{\Z /p}_{\;*}^G X$ can be calculated on the chain level.
Let $C^G(X)$ be the cellular chain complex of $X$ in the category of $G$-coefficient systems in the sense of Bredon \cite{Bredon},
i.e. functors $\mathscr{O}_G^{Op} \rightarrow Ab$ where $\mathscr{O}_G$ is the orbit category.
This is defined by
$$(C_n^G(X))(G/H) = C_n^{\text{cell}}(X^H).$$
Then we have
$$H\underline{A}_n(X) = H_n (C^G(X) \otimes_{\mathscr{O}_G} \underline{A})$$
where $\underline{A}$ is the constant co-coefficient system for an abelian group $A$,
i.e. the functor $\mathscr{O}_G \rightarrow Ab$ where for $f: G/H \rightarrow G/K \in Mor(\mathscr{O}_G)$, $f_*$ is multiplication by
$\frac{|K|}{|H|}$.

We will show that
\beg{CollapseTheoremStep}{C^G (X) \otimes_{\mathscr{O}_G} \underline{\Z /p}\cong \bigoplus_{(H)} \widetilde{C}^{cell}(X^H / \bigcup_{H\subsetneq K} X^K)\otimes_{\Z [W(H)]} \Z /p.}
In each degree separately, \rref{CollapseTheoremStep} holds as abelian groups, since all $\mathscr{O}_G$-identifications corresponding to non-isomorphisms are trivial.
For any $f:H\subsetneq K$, consider the summand of the differential $d^{tot}$ of $C^G (X) \otimes_{\mathscr{O}_G} \underline{\Z/p}$
$$d_{H,K}: C_n(X^H) \otimes \Z /p \rightarrow C_{n-1}(X^K) \otimes \Z /p.$$
By conjugation, it suffices to show that these maps are $0$.

The differential $d^{tot}$ is given by
$$\bigoplus d/ \sim: \bigoplus C_n(X^H) \otimes \Z/p /\sim \rightarrow \bigoplus C_{n-1} (X^H) \otimes \Z /p/\sim$$
where $\sim$ denotes the equivalence relation generated by 
$$f^* a\otimes b \sim a\otimes f_*b.$$
In particular, for $q\in C_n(X^H)$, let $c\in C_{n-1} (X^H)$ be the sum of the terms of $d (q)$ on cells in $X^K$, where $d$ is the differential of $C(X^H)$.
Then we have
$$d_{H,K}(q) = f^* c\otimes 1 = c\otimes f_* 1 = c\otimes \frac{|K|}{|H|} =0.$$
Therefore, $d_{H,K}=  0$.
Thus, we have proved \rref{CollapseTheoremStep}, and hence the spectral sequence collapses to $E^1$.

\end{proof}


\vspace{5mm}

\section{Geometric Fixed Points}\label{ProofsMainResult}

\vspace{5mm}

In this section, we shall apply the methods of the previous section to completely calculate the coefficients of the geometric fixed point spectrum
$$\Phi_*^G H\underline{A} = \widetilde{H\underline{A}}_* \widetilde{E\mathscr{F} [G]}$$
where $\mathscr{F}[G]= \{ H | H\subsetneq G\}$ for any finite group $G$.

A basic fact about posets is useful for computing examples:

For functors $F:\mathscr{C} \rightarrow \mathscr{D}$ between any categories, we get continuous maps
$$|F| : |\mathscr{C}| \rightarrow |\mathscr{D}|,$$
and for natural transformations $F \rightarrow G$, we get
$$|F| \simeq |G|.$$
In particular, if $f , g : P \rightarrow Q$ are morphisms of posets and $f\leq g$, then $|f| \simeq |g|$.
Thus, in particular, if $P$ has lowest or highest element $x$, $|Id| \simeq |Const_x|$. Therefore, then, $|P|$ is contractible.

\begin{lemma}\label{POSETLemma}
For posets $Q\subseteq P$, if there exists a morphism $f: P \rightarrow Q$ which satisfies both
\begin{enumerate}
\item For every $x \leq y\in P $, we have $f(x) \leq f(y)$

\item For every $x\in P$, we have $x\leq f(x)$ (or alternately $x\geq f(x)$),

\end{enumerate}
then the inclusion induces a homotopy equivalence $|Q|\simeq |P|$.
\end{lemma}

\begin{proof}
Suppose we have $P, Q$  posets and a $f: P\rightarrow Q$ satisfying the assumptions.
We have an inclusion
$$\iota : Q \hookrightarrow P.$$
For an $x\in Q$, $\iota (x) = x$. So, for every $x\in Q$
$$f \circ\iota (x) = f(x) \geq x = Id_Q (x).$$
Therefore $f\circ\iota \geq Id_Q$. So $|f| |\iota| = |f\circ \iota| \simeq Id_{|Q|}$.
On the other hand, for $x\in P$, $f(x) \in Q$, so
$$\iota \circ f (x) = f(x) \geq x = Id_P (x).$$
So, $ f \circ \iota \geq Id_P$. So $|\iota | |f| = |\iota \circ f| \simeq Id_{|P|}$.

\end{proof}

Denote by $G_p'$ the Frattini subgroup of $G$, i.e. the subgroup generated by the commutator subgroup and $p$'th powers.

\begin{proposition}\label{abelianpgp}
Suppose $G$ is a $p$-group. Then for any $G$-spectrum $E$, we have
$$\Phi^G (E) \simeq \Phi^{G^{ab}/p}(E^{G_p'}).$$
In particular, for a constant Mackey functor $\underline{A}$, we have
$$\Phi^G (\underline{A}) \simeq \Phi^{G^{ab} /p}(\underline{A}).$$
\end{proposition}

\noindent {\bf Comment:}
Note that one always has an equivalence 
$$\Phi^G (E) \simeq \Phi^{G/H} \Phi^H E.$$
The special feature here is that $E^H \rightarrow \Phi^H E$ induces an equivalence on $G/H$-geometric
fixed points if $G$ is a $p$-group and $H$ is the Frattini subgroup.

\vspace{5mm}

We shall first prove

\begin{lemma} 
Let $G$ be a $p$-group and let $H\subseteq G$ be a subgroup not containing $G'_p$. Then
$$|P_H^{\mathscr{F}[G]}| \simeq *.$$
\end{lemma}

\begin{proof}
Denote by $Q$ the poset of proper subgroups of $G$ containing $G_p' \cdot H$. 
We have an inclusion $Q \subseteq P_H^{\mathscr{F}[G]}$.
By the Burnside basis theorem, for any subgroup $K\subsetneq G$, we have $K\cdot G_p' \subsetneq G$.
Thus, we have a map of posets $\varphi: P_H^{\mathscr{F}[G]} \rightarrow Q$ given by
$$K\mapsto K\cdot G_p'.$$
Also, for $K\in Q$, $\varphi (K) \supseteq K$.
Thus, $|Q|\simeq |P_H^{\mathscr{F}[G]}|$ by Lemma \ref{POSETLemma}.

However, $|Q|\simeq *$ since $Q$ has a minimal element.

\end{proof}

Note that this implies Proposition \ref{abelianpgp}, since the quotient map
\beg{QuotientMapforAbelianProp}{\widetilde{E\mathscr{F}}[G]\rightarrow \widetilde{E\mathscr{F}}[G/G_p']}
induces an isomorphism on $E^1$-terms of the spectral sequence \rref{spectralsequencebased}.
(The spectral sequence is not functorial in general with respect to change of groups.
In the present case, however, we have a surjection of groups which preserves
the height of the proper subgroup containing $G_p'$, while the remaining terms in the source are $0$.
Thus, a morphism of spectral sequences which is an isomorphism an $E^1$ arises.)

\vspace{3mm}

Also note that for the present purpose, the spectral sequence can be skipped entirely and one can simply argue that
\rref{QuotientMapforAbelianProp} is an equivalence by examining its $H$-fixed points for each $H$:
The fixed point set of the left hand side is contractible for $H\subsetneq G$ and equivalent to $S^0$ if $H=G$, while
the right hand side is contractible if $H\cdot G_p' \subsetneq G$ and equivalent to $S^0$ if $H\cdot G_p' = G$.
By the Burnside basis theorem, both conditions are equivalent.
This was pointed out to me during the process of revising this paper.

\begin{proposition}\label{nonabelianpgp}
If $G$ is not a $p$-group, then
$$\Phi^G (\underline{A}) = 0.$$
\end{proposition}



\vspace{5mm}

\begin{proof}

First, suppose $G$ is a finite group that is not a $p$-group.

The spectrum $\Phi^G (\underline{\Z})$ is a commutative ring spectrum, since we have 
$$S^{\infty V} \wedge S^{\infty V}  \cong S^{\infty V}.$$
Choose a prime $p$. Then by the first Sylow theorem, there exists a $p$-Sylow subgroup $P$ of $G$.
Then there exists a $H$ with $P\subseteq H \subsetneq G$ that is maximal (i.e. there does not exist $K$ such that
$P\subseteq H \subsetneq K \subsetneq G$).
Therefore the contribution of $H$ to the spectral sequence will include
\beg{CstMackeyFunctorThmHContribution}{\widetilde{H}_0^{W(H)}(\widetilde{\emptyset})\simeq H_0^{W(H)}(*) = \Z.}
In \rref{CstMackeyFunctorThmHContribution}, $1\in \Z$ represents an element $\eta \in E^1_{1,0}$ where
$$d^1(\eta) = \pm \frac{|G|}{|H|} \in \Z = E^1_{0,0}.$$
Since we have $p \nmid \frac{|G|}{|H|}$, the g.c.d. of all these numbers is $1$, and thus, 
$1\in \Z = E^1_{0,0}$ of the spectral sequence \rref{Specialspectralsequencebased} is killed. Since $\Phi^G_* (\underline{\Z})$ is a commutative ring, it must be $0$.
So
$$\Phi^G(\underline{\Z}) =0.$$
\end{proof}

Now, one can apply the results of my previous paper \cite{SophieKrizEquivariant}, as well as those of  \cite{HollerIgorKrizCoefficients, HollerIgorKrizOrdinary}, to calculate $\Phi^{(\Z/ p )^n} (\underline{\Z})$. 

Recall that for any space or spectrum $X$, we can obtain maps
\beg{Bockstein1}{
\diagram
H^n(X; \Z/p) \rto^\beta & H^{n+1}(X;\Z/p)\\
\enddiagram
}
\beg{Bockstein2}{
\diagram
H^n (X;\Z/p) \rto^\beta &H^{n+1}(X;\Z) \\
\enddiagram
}
as the connecting maps of the long exact sequences from taking cohomology with coefficients in the following respective short exact sequences:
$$0\rightarrow \Z/p \rightarrow \Z / (p^2) \rightarrow \Z/p \rightarrow 0$$
$$
\diagram
0\rto &\Z \rto^{\cdot p} & \Z \rto & \Z/p \rto & 0. \\
\enddiagram
$$
These are called the Bockstein maps, and the long exact sequence involving \rref{Bockstein2}
forms an exact couple which gives rise to the Bockstein spectral sequence, in which \rref{Bockstein1} is $d^1$.

Consider first the case of $p>2$. Recalling the notation of Theorems \ref{Prepretheorem} and \ref{Pretheorem}, the Bockstein acts by
$$\beta (dz_i) = z_i$$
$$\beta z_i = 0.$$
Also recall that we have
\beg{BMultiplication}{\beta(ab) = \beta a \cdot b + (-1)^{|a|} a \cdot \beta b.}

Thus we get
$$ \beta( t_\alpha ) = 0$$
$$ \beta (u_\alpha) = 1$$
(Note that
$\beta$ preserves the relations of \rref{oldresult}. For example,
$$\beta(-u_\beta u_{\alpha + \beta} + u_{\alpha} u_{\alpha + \beta} - u_{\alpha} u_{\beta})=$$
$$= - \beta(u_\beta) u_{\alpha + \beta} + u_\beta \beta(u_{\alpha + \beta}) + \beta(u_\alpha) u_{\alpha + \beta}$$
$$- u_\alpha \beta(u_{\alpha + \beta}) - \beta(u_{\alpha} ) u_\beta + u_\alpha \beta(u_\beta) = 0.)$$

Now by computing directly on the chain level the
equivariant $(\Z /p)^n$-homology of $S^{\infty V}$ where $V$ is the reduced regular representation,
the standard $(\Z /p)^n$-CW decomposition of $S^{\infty V}$ (thought of as a colimit of smash products of representation spheres of
non-trivial $1$-dimensional complex representations $\alpha$) has only one $0$-cell beside the base point, to which (using one $\alpha$)
$p$ $1$-cells are attached on which $(\Z /p)^n$ acts transitively.
Thus, $\Phi_0^{(\Z /p)^n} (\underline{\Z}) = \Z /p$.
Since $\Phi_*^{(\Z/p)^n} (\underline{\Z})$ is a ring, it has characteristic $p$ (i.e. every element is annihilated by $p$).
Thus the Bockstein spectral sequence collapses to $E^2$, or in other words
\beg{BSSCollapses}{H_* (\Phi_*^{(\Z/ p )^n} (\underline{\Z/p}), \beta) =0.}
Hence, we have an exact sequence
\beg{ExactSequence}{
\diagram
0 \rto & \Phi_*^{(\Z/p)^n} (\underline{\Z}) \rto & \Phi_*^{(\Z/ p )^n} (\underline{\Z/p}) \rto^\beta & \Phi_*^{(\Z/ p )^n} (\underline{\Z/p}).\\
\enddiagram
}

Therefore, $\Phi_*^{(\Z/ p )^n} (\underline{\Z})$ contains the elements $t_\alpha$ and $\sum_i a_i u_{\alpha_i}$ where $\sum_i a_i = 0 \in \Z/p$.
Choosing an $\alpha_0 \in (\Z / p )^n \smallsetminus \{0 \}$, since we are in characteristic $p$, the elements $\sum_i a_i u_{\alpha_i}$ where $\sum_i a_i = 0 \in \Z/p$ are linear combinations of $\widetilde{u}_\alpha = u_\alpha - u_{\alpha_0}$.
One easily verifies the relations \rref{relationspg2}. For example,
$$-\widetilde{u}_\beta \widetilde{u}_{\alpha+\beta} + \widetilde{u}_\alpha \widetilde{u}_{\alpha +\beta} - \widetilde{u}_\alpha \widetilde{u}_\beta=$$
$$=-(u_\beta - u_{\alpha_0}) (u_{\alpha + \beta}-u_{\alpha_0}) + (u_\alpha -u_{\alpha_0}) (u_{\alpha + \beta}-u_{\alpha_0})$$ 
$$-(u_\alpha -u_{\alpha_0})(u_\beta -u_{\alpha_0}) = - u_\beta u_{\alpha + \beta} +u_{\alpha_0} u_{\alpha + \beta} + u_\beta u_{\alpha_0}- u^2_{\alpha_0} + u_\alpha u_{\alpha + \beta}$$
$$- u_{\alpha_0}u_{\alpha + \beta}-u_\alpha u_{\alpha_0}+ u_{\alpha_0}^2 $$
$$- u_\alpha u_\beta + u_{\alpha_0} u_\beta + u_\alpha u_{\alpha_0} -u_{\alpha_0}^2= 0$$
Let $\widetilde{R}_n$ denote the ring $\Z / p [t_\alpha, \widetilde{u}_\alpha]$ modulo the relations \rref{relationspg2}.
Write \rref{ExactSequence} as
$$ 
\diagram
0\rto & R_{\Z} \rto & R_{\Z /p } \rto^\beta & R_{\Z / p} \\
\enddiagram
$$
We therefore have a homomorphism of rings 
$$\varphi: \widetilde{R}_n \rightarrow R_\Z.$$
We want to prove that this is an isomorphism.

\vspace{3mm}

Now, let us consider $p=2$.
Choose again a representative $\alpha_0 \in (\Z/2)^n \smallsetminus \{0\}$.
Again, the $R_\Z$ contains elements of the form $\sum_i a_i y_i$ with $\sum_i a_i = 0$ which are generated
by $\widetilde{y}_\alpha = y_\alpha - y_{\alpha_0}$.
Also, the elements $t_\alpha = y_\alpha^2 \in R_\Z$ (by \rref{BMultiplication}), but note that
$\widetilde{y}_\alpha^2 = y_\alpha^2 + y_{\alpha_0}^2$ (since we are in characteristic $2$), so we only need to include $t_{\alpha_0}$ in the generators.
Now similarly as for $p>2$, one proves the relations.
\beg{RelationsPeq2}{\begin{array}{c}
\widetilde{y}_{\alpha_0} = 0\\
\widetilde{y}_{\alpha} \widetilde{y}_\beta + \widetilde{y}_\alpha \widetilde{y}_{\alpha + \beta} +\widetilde{y}_\beta \widetilde{y}_{\alpha + \beta} + t_{\alpha_0} =0.\\
\end{array}
}
Let
$\widetilde{R}_n$ denote the quotient of the ring
$\Z / 2 [\widetilde{y}_\alpha, t_{\alpha_0}] $ 
by relations \rref{RelationsPeq2}.

Again, we have a homomorphism of rings
$$\varphi: \widetilde{R}_n \rightarrow R_{\Z}$$
$$\widetilde{y}_\alpha \mapsto y_\alpha - y_{\alpha_0},$$
$$t_{\alpha_0} \mapsto y_{\alpha_0}^2.$$
We can prove that $\varphi$ is an isomorphism by calculating the Poincar\'{e} series of $\widetilde{R}_n$, checking
that it is the same as the Poincar\'{e} series of $R_\Z$ and exhibiting an additive basis of $\widetilde{R}_n$ that is linearly independent in $R_{\Z/2}$.

Recall from \cite{ HollerIgorKrizCoefficients, HollerIgorKrizOrdinary} that the Poincar\'{e} series of $R_{\Z/p}$ is
$$P (R_{\Z/p}) = \frac{1}{ (1-x)^{n}} \prod_{i=1}^n (1+ (p^{i-1} -1)x).$$
Thus, by \rref{BSSCollapses},
we have an exact sequence of graded $\Z /p$-vector spaces
$$
\diagram
0 \rto & R_{\Z} \rto & R_{\Z / p} \rto^{\beta} & R_{\Z /p}[1] \rto^{\beta} & R_{\Z} [2] \rto & 0\\
\enddiagram
$$
and hence
\beg{PoincareSeriesZVSZp}{P(R_\Z) = \frac{1}{1+x} P (R_{\Z/p}) = \frac{1}{(1-x^2) (1-x)^{n-1}} \prod_{i=1}^n (1+ (p^{i-1} -1)x).}
For $p=2$, we know that the Poincar\'{e} series of $\widetilde{R}_1$ is $\frac{1}{1-x^2}$.
Now we can treat the $\alpha$ as elements of $(\Z /2)^n\smallsetminus \{ 0 \}$ and assume that $\alpha_0 = (1, 0, \dots, 0)$.
For $n=2$, we only have $\widetilde{y}_{(0,1)}$, $\widetilde{y}_{(1,1)}$, $\widetilde{y}_{(1,0)}= 0 $.
By the relations, we have $\widetilde{y}_{(0,1)} \widetilde{y}_{(1,1)} = t_{\alpha_0}$.
Therefore this ring has additive basis $\widetilde{y}_{(0,1)}^{m\geq 1} t_{\alpha_0}^{m'}$, $\widetilde{y}_{(1,1)}^{k\geq 1} t_{\alpha_0}^{k'}$, and
$t_{\alpha_0}^\ell$ ($m', k',\ell\geq 0$), which give the terms $\frac{x}{(1-x)(1-x^2)}$ twice and $\frac{1}{1-x^2}$ in the Poincar\'{e} series.
Therefore, the Poincar\'{e} series of the ring is 
$$P (\widetilde{R}_2) = \frac{1}{1-x^2} + 2\cdot \frac{x}{(1-x)(1-x^2)} = \frac{1+x}{(1-x)(1-x^2)}.$$
After this, we can continue by induction since the relations imply
$$P(\widetilde{R}_n) = P (\widetilde{R}_{n-1}) \cdot (1 + (2^{n-1}-1 )x) \cdot \frac{1}{1-x}$$
similarly as in \cite{HollerIgorKrizCoefficients}:
The additive basis is formed by the additive basis of $\widetilde{R}_{n-1}$ times $\widetilde{y}_{(0,\dots, 0,1)}^{\geq 0}$ or
$\widetilde{y}^{\geq 1}_{(\alpha', 1)}$ where $\alpha'\in (\Z/2)^{n-1} \smallsetminus \{ 0\}$.
These elements are linearly independent in $R_{\Z /2}$ by performing a similar induction there (which was done in \cite{HollerIgorKrizCoefficients}).

\vspace{5mm}

The case of $p>2$ is completely analogous. We define the homomorphism of rings
$$\varphi: \widetilde{R}_n \rightarrow R_{\Z}$$
$$\widetilde{u}_{\alpha} \mapsto u_{\alpha} - u_{\alpha_0}$$
$$t_{\alpha} \mapsto t_{\alpha},$$
which again we can check is a ring homomorphism by computing the relations in the target. 
Again, we check $\varphi$ is an isomorphism by checking that for every n, the Poincar\'{e} series of $\widetilde{R}_n$ agrees with \rref{PoincareSeriesZVSZp}.
We have $\widetilde{R}_1$ is generated by the basis of $(t_{(1)})^{m\geq 0}$. Thus, as before, $P(\widetilde{R}_1) = \frac{1}{1-x^2}$.
This time, for every $n\geq 2$, an additive basis of $\widetilde{R}_n$ is given by the additive basis of $\widetilde{R}_{n-1}$ times
$t_{(0,\dots,0,1)}^{\geq 0} \cdot \widetilde{u}_{(0,\dots,0,1)}^{\epsilon}$ where $\epsilon \in \{0,1\}$, or times $t_{(\alpha',1)}^{\geq 1}$, or times $t_{(\alpha',1)}^{\geq 0} \cdot \widetilde{u}_{(\alpha',1)}$ where $\alpha' \in (\Z / p)^{n-1} \smallsetminus \{ 0\}$. 
This gives
$$P(\widetilde{R}_n) = P (\widetilde{R}_{n-1}) \cdot (1 + (p^{n-1}-1 )x) \cdot \frac{1}{1-x},$$
and we can proceed by induction.

\vspace{5mm}

\section{Another Example}\label{ExtraSpecialExample}

\vspace{5mm}

For the rest of the paper, we will consider equivariant homology with constant coefficients $\underline{\Z /p}$ for a prime $p$.
If $G=(\Z/p)^n$ is an elementary abelian group, $S$ is a set of $1$-dimensional representations (real or complex
depending on whether $p=2$ or $p>2$), and $\gamma = \bigoplus_{\alpha \in S} \alpha$, then we completely calculated in \cite{SophieKrizEquivariant} the ($\Z$-graded) coefficients of
\beg{Calculatethecoefficientsof}{\widetilde{H\underline{\Z / p}}^G_{\;*} (S^{\infty\gamma}) = \widetilde{H\underline{\Z/p}}_{\;*}^G \widetilde{E\mathscr{F}_{\gamma}}.}
By the above method, for any $p$-group $G$ and any set $S$ of non-trivial irreducible $1$-dimensional representations of $G/ G_p'$,
$\gamma = \bigoplus_{\alpha \in S} \alpha$, we have
$$\widetilde{H\underline{\Z / p}}^G_{\;*} (S^{\infty\gamma}) = \widetilde{H\underline{\Z/p}}_{\;*}^{G/G'_p} (S^{\infty \gamma}).$$

However, for a $G$-representation $\gamma$ which does not factor through $G/G_p'$ the calculation of
\rref{Calculatethecoefficientsof}
can be non-trivial.
In this section, as an example, we consider the case where $G$ is the split extraspecial group (as described below) at $p=2$ and $V$ is the
irreducible real representation non-trivial on the center.
This group is the central product of $n$ copies of $D_8$.
We write $V_n = (\Z/2 \oplus \Z/2)^n$
where the generators of the $i$th copy of $\Z/2 \oplus \Z/2$ are denoted by $v_{1,i}, v_{2,i}$.
Define $q(v_{1,i}) = q(v_{2,i}) =0$, $q(v_{1,i} + v_{2,i}) =1$, and let $q$ be additive between different $i$ summands.
This is a split quadratic form on the $\mathbb{F}_2$-vector space $V_n$ with associated symplectic form
$$b(x,y) = q(x+y) +q(x)+q(y).$$
A vector subspace $W\subseteq V$ is called {\em isotropic} when $b$ is $0$ on $W$ and is called {\em $q$-isotropic} if $q$ is $0$ on $W$.
The split extraspecial group $\widetilde{V}_n$ is an extension
$$1\rightarrow \Z /2 \rightarrow \widetilde{V}_n \rightarrow V_n \rightarrow 1$$
where for $v\in V_n$, $2v \neq 0$ if and only if $q(v) \neq 0$ and for $v, w \in V_n$, $vwv^{-1}w^{-1} = b(v,w)$.

Clearly $\widetilde{V}_n$ is isomorphic to a central product of $n$ copies of $D_8$ and the (real) irreducible representation $\gamma$
non-trivial on the center is obtained as the tensor product of the dihedral representation of the $n$ copies of $D_8$.
We shall apply Proposition \ref{spectralsequence} to
$$\widetilde{H\underline{\Z /2}}_{\;*}^{\widetilde{V}_n} (S^{\infty \gamma}).$$
The family $\mathscr{F}_{\gamma}$ consists of elementary abelian subgroups of $\widetilde{V}_n$ disjoint with the center which project to
$q$-isotropic subspaces of $V_n$.
Two subgroups are conjugate if and only if they project to the same $q$-isotropic subspaces $U$ of $V_n$.
We shall refer to these subgroups as {\em decorations} of $U$, and call them {\em decorated} $q$-isotropic subspaces.

\vspace{5mm}

We shall need to consider the following modular representations of the split extraspecial $2$-group $\widetilde{V}_n$:
All these representations will factor through the Frattini quotient $V_n$.
By the representation $\underline{2}_i$, we mean a tensor product of the regular representation of $\Z /2 \{ v_{1,i}\}$
with the trivial representation on $\Z /2 \{ v_{2, i}\}$, where $v_{1,j}, v_{2,j}$ act trivially for $j\neq i$. (For counting purposes, equivalently, $1$ and $2$ can be reversed).
The representation $\underline{3}_i$ is the kernel of the augmentation from the regular representation on $\Z/2\{ v_{1,i}, v_{2,i}\}$
to the  trivial representation (with the other coordinates also acting trivially).

Let $\underline{P}_n$ be the poset of elementary abelian subgroups of $\widetilde{V}_n$ which project to a non-trivial $q$-isotropic
subspace of $V_n$.
We will refer to these subgroups as {\em decorated $q$-isotropic subspaces} of $V_n$.

\vspace{5mm}

\begin{theorem}
For $n>1$, $\widetilde{H}_k (|\bar{P_n}|) =0$ except for $k=n-1$.
As a $\widetilde{V}_n$-representation, $\mathscr{H}_n := \widetilde{H}_{n-1} (|\overline{P}_n|)$
is given recursively as follows:
$$\mathscr{H}_1 = \underline{3}_1$$
\beg{RecursiveForEqHomOfPoset}{
\begin{array}{c}
\mathscr{H}_{n+1} = \underline{3}_{n+1} \otimes \mathscr{H}_n \\ [1ex]
 \oplus \underline{2}_{n+1} \otimes (2^{2n-1} - 2^{n-1}) \mathscr{H}_n\\ [1ex]
\oplus \underline{2}_{n+1} \otimes (2^{2n-1} + 2^{n-1} -1)(2\mathscr{H}_n - \underline{2}_n 2^{2n-2}\mathscr{H}_{n-1}).
\end{array}}
The subtraction in the last term is to be interpreted recursively as follows:
We set $\mathscr{H}_0 =1$.
Then one copy of $\underline{2}_n\mathscr{H}_{n-1}$ is ``subtracted" from the first summand of \rref{RecursiveForEqHomOfPoset} with $n$ replaced $n-1$,
to leave a copy of $\mathscr{H}_{n-1}$.
The remaining $2^{2n-2}-1$ copies of $\underline{2}_n \mathscr{H}_{n-1}$ are subtracted from the second summand
of the formula \rref{RecursiveForEqHomOfPoset} with $n$ replaced by $n-1$ (thus, only one copy of the $2\mathscr{H}_n$ is involved in the subtraction).
\end{theorem}

\vspace{5mm}

\noindent {\bf Comment:}
The expression for $\mathscr{H}_n$ given by the Theorem is a direct sum of representations of the form
$$\bigotimes_{i\in S_1} \underline{2}_i \otimes \bigotimes_{j\in S_2} \underline{3}_j$$
for $S_1 \cap S_2 = \emptyset$. I do not know if homology groups of $\widetilde{V}_n$ with coefficients in these representations
are all completely known. For $|S_2|\leq 2$, they can be deduced from the computation of Quillen \cite{QuillenExtraSpecialGroups}.

\vspace{5mm}

\begin{proof}
We proceed by induction on $n$. In the case of $n=1$, the isotropic subspaces are $\langle v_1\rangle$
and $\langle v_2\rangle$, and there are two lifts to $\widetilde{V}_1$, and each pair of lifts is given by one $\Z /2$-summand
of $V_1 = \Z /2\oplus \Z /2$. Note that we have $\underline{3}$, because we must consider reduced homology.

Now to pass from $|\bar{P}_n|$ to $|\bar{P}_{n+1}|$, consider first the poset $\bar{Q}_n$ of decorated $q$-isotropic subspaces of
$V_{n+1}$ which intersect non-trivially with $V_n$.
Then the inclusion $|\bar{P}_n| \subset |\bar{Q}_n|$ is an equivalence by Lemma \ref{POSETLemma}, considering the map the other way given by
$$W\mapsto W\cap V_n.$$
Now a $q$-isotropic subspace $W$ of $V_{n+1}$ with $W\cap V_n =0$ has $dim(W) \leq 2$.
Let $\bar{R}_{n+1}$ denote the union of $Q_n$ and the set of decorated isotropic subspaces $W\subset V_{n+1}$ with $dim(W)=2$,
$W\cap V_n =0$.
Then the poset $(\bar{R}_{n+1})_{\geq W} \cap Q_n$ for any such space $W$ consists of
copies of the poset $\underline{4} \bar{P}_{n-1}$ (here we use $\underline{4}$ to denote $4$ additional independent decorations).
The factors \rref{RecursiveForEqHomOfPoset} correspond to the choices of $W$, consisting of one non-zero $q$-isotropic vector
$w\in V_n$, and one vector $w'$ with $q(w')=1$, $b(w, w')=1$.
There are $2^{2n-2}$ choices of $w'$ for each $w$.
Then $W = \langle v_1 + w, v_1+v_2+w'\rangle$. This leads to a based cofibration
\beg{CofibrESGroupsPoset}{|\bar{P}_n| \rightarrow |\bar{R}_{n+1}| \rightarrow \operatornamewithlimits{\bigvee}_{\underline{4} (2^{2n-1} + 2^{n-1} -1)2^{2n-2}} \Sigma |\bar{P}_{n-1}|.}
Now $\bar{P}_{n+1}$ is the union of $\bar{R}_{n+1}$ with the set of all decorated $q$-isotropic subspaces $W\subset V_{n+1}$ with
$W\cap V_n =0$, $dim(W) =1$.
For such a space $W$, $(\bar{P}_{n+1})_{\geq W} \cap R_{n+1}$ consists of
\beg{TwoChooseIsoTwice1Once}{\underline{2} (2 (2^{2n-1} + 2^{n-1}) + (2^{2n-1} -2^{n-1}))}
copies of $\bar{P}_n$.
The $\underline{2}$, again, corresponds to additional decorations. The $2(2^{2n-1} + 2^{n-1})$ summands correspond to $q$-isotropic vectors
of the form $w+v_1, w+v_2$, for $w\in V_n$, the $2^{2n-1} -2^{n-1}$ summand correspond to $q$-isotropic vectors of the form
$w+v_1+v_2$. Thus, we obtain a based cofibration sequence
\beg{CofibrESGroupsPosetNextStep}{|\bar{R}_{n+1}| \rightarrow |\bar{P}_{n+1}| \rightarrow \operatornamewithlimits{\bigvee}_{\underline{2} (2(2^{2n-1} + 2^{n-1}) + (2^{2n-1} -2^{n-1}))} \Sigma |\bar{P}_n|.} 

Now from \rref{CofibrESGroupsPoset} and \rref{CofibrESGroupsPosetNextStep}, we can easily eliminate $|\bar{R}_{n+1}|$, as we see that the copies of
$\Sigma |\bar{P}_n|$ in \rref{CofibrESGroupsPosetNextStep} corresponding to $w=0$ project identically to $\Sigma |\bar{P}_n| \subset \Sigma |\bar{R}_{n+1}|$ under the connecting map.
Thus, we obtain a cofibration sequence of the form
\beg{FinalCofibrESGroupsPoset}{|\bar{P}_{n+1}| \rightarrow \operatornamewithlimits{\bigvee}_{\underline{2}(2(2^{2n-1} + 2^{n-1} -1) + 2^{2n-1} -2^{n-1})} \hspace{-5mm}\Sigma |\bar{P}_n|\rightarrow \operatornamewithlimits{\bigvee}_{\underline{4} (2^{2n-1} + 2^{n-1} -1) 2^{2n-2}}\Sigma^2|\bar{P}_{n-1}|.}
The second map \rref{FinalCofibrESGroupsPoset} is shown to be onto in reduced homology using the sums of terms indicated in the statement of the Theorem.
(In particular, we consider, for a $q$-isotropic vector $w+v$, with $0\neq w \in V_{n+1}$, all choices of vectors $w'$ such that $\langle w+v_1, w'+ v_1 +v_2\rangle$ is $q$-isotropic. Note that this includes but is not equal to, for $n>1$, all $\langle w+v_1, u\rangle$ $q$-isotropic.)

The dichotomy between canceling the first or second summand in \rref{RecursiveForEqHomOfPoset} comes from distinguishing whether the projection of $w$ to $V_{n-1} $ is $0$ or not.

\end{proof}

\noindent {\bf Comment:}
The same method shows that the reduced homology of the poset of undecorated $q$-isotropic subspaces of $V_n$ is concentrated in degree $n-1$
and has rank $2^{n(n-1)}$.
This poset (for $n\geq 2$) is the Tits building of $\Omega^+_{2n} (2)$
(the adjoint Chevally group of type $D_n$ at the prime $2$), and this fact therefore also follows from the Solomon-Tits Theorem \cite{Solomon}.

\vspace{5mm}

The number of $q$-isotropic subspaces $U_k$ of dimension $k$ of $V_{n}$ is
$$\resizebox{0.9\hsize}{!}{$v_{n,k}=2^{\frac{k(k-1)}{2}}\cdot\frac{(2^{2n-1} +2^{n-1} -1) (2^{2n-3} +2^{n-2} -1) \dots (2^{2n-2k+1} +2^{n-k}-1)}
{(2^k-1)(2^{k-1}-1) \dots (2-1)}$}$$
The Weyl group of $U_k$ is $\widetilde{V}_{n-k}$. Thus, we have proved the following
\begin{theorem}
We have
$$\widetilde{H\underline{\Z /2}}_{\;0}^G (S^{\infty \gamma}) = \Z/2.$$
For $i>0$,
$$\widetilde{H\underline{\Z /2}}_{\;i}^G (S^{\infty \gamma}) = v_{n,k} \bigoplus_{k=0}^n H_{i-n+k-1}(\widetilde{V}_{n-k}, \mathscr{H}_{n-k}).$$
\end{theorem}
\qed

\end{document}